\documentclass[11pt]{article}
\usepackage[T1]{fontenc}
\usepackage{amsmath}%
\usepackage{amsthm}
\usepackage{color,hyperref}
\usepackage{amsxtra}%
\usepackage{enumitem}
\usepackage{amsfonts}%
\usepackage{amssymb}%
\usepackage{graphicx}
\usepackage{subfigure}
\usepackage[margin=3cm]{geometry}
\usepackage{cite}
\usepackage{array}
\usepackage{booktabs}
\newcommand{\R}{\mathbb{R}}

\newcommand{\Z}{\mathbb{Z}}

\renewcommand{\phi}{\varphi}

\renewcommand{\subset}{\subseteq}

\theoremstyle{plain}
\newtheorem{theorem}{Theorem}
\newtheorem*{theorem*}{Theorem}

\newtheorem{corollary}[theorem]{Corollary}
\newtheorem{lemma}[theorem]{Lemma}

\theoremstyle{definition}

\newtheorem{definition}[theorem]{Definition}

\theoremstyle{remark}

\numberwithin{equation}{section}
\numberwithin{theorem}{section}
\graphicspath{{images/}}

\newcommand{\lto}{\to}
\newcommand{\inn}{\not\in }
\newcommand{\dqki}{\nabla^{k,-}_{i}}

\newcommand{\dqi}[1]{\nabla^{#1,-}_{i}}

\newcommand{\ind}{\ \ \text{in } }

\begin{document}

\title{High-order  filtered  schemes for the  Hamilton-Jacobi continuum limit of  nondominated sorting\thanks{The research described in this paper was partially supported by NSF grant DMS-1713691  and a  University of Minnesota UROP award.}}

\author{Warut Thawinrak\thanks{Department of Mathematics, University of  Minnesota ({\tt thawi001@umn.edu  })}\and Jeff Calder\thanks{Department of Mathematics, University of  Minnesota ({\tt jcalder@umn.edu})}}
\maketitle
\begin{abstract}
We investigate high-order finite difference schemes for the Hamilton-Jacobi equation continuum limit of nondominated sorting. Nondominated sorting is an algorithm for sorting points in Euclidean space into layers by repeatedly removing minimal elements. It is widely used in multi-objective optimization, which finds applications in many scientific and engineering contexts, including machine learning. In this paper, we show how to construct filtered schemes, which combine high order possibly unstable schemes with first order monotone schemes in a way that guarantees stability and convergence while enjoying the additional accuracy of the higher order scheme in regions where the solution is smooth. We prove that our filtered schemes are stable and converge to the viscosity solution of the Hamilton-Jacobi equation, and we provide numerical simulations to investigate the rate of convergence of the new schemes.
\end{abstract}

\section{Introduction}
\label{sec:intro}

In this paper, we  investigate  high-order finite difference schemes for the Hamilton-Jacobi equation
\begin{equation}\label{P1}
\left.
\begin{aligned}
u_{x_1} \cdots u_{x_n} &= f& &\text{in } (0,1]^n\\
u &=0& &\text{on } \partial [0,1]^n \setminus (0,1]^n,
\end{aligned}\right\}
\end{equation}
where  $f\ge 0$.
This Hamilton-Jacobi equation is the continuum limit of nondominated sorting, which is an algorithm for arranging points in Euclidean space into layers by peeling away extremal points. Nondominated sorting is fundamental in multi-objective optimization problems, which are ubiquitous in science and engineering, and more recently in machine learning. For more details on the connection to nondominated sorting and applications, we refer the reader to \cite{calder2014,calder2015PDE,calder2017,calder2016direct,hsiao2015,hsiao2015b,hsiao2012,deb2002,deuschel1995,hammersley1972}.

The Hamilton-Jacobi equation \eqref{P1} has a unique non-decreasing viscosity solution. In order to select the viscosity solution of \eqref{P1}, the finite difference scheme is required to be monotone~\cite{barles1991}. Roughly speaking, the monotonicity property leads to a maximum principle for the scheme, which is one of the main techniques for proving stability and ensuring convergence to the viscosity solution. Unfortunately, all monotone schemes are necessarily first order at best~\cite{caldernotes}. 

It has been observed~\cite{barles1991,oberman2015filtered} that the monotonicity property can be relaxed to hold only approximately, with a residual error that vanishes as the grid is refined, while still ensuring the scheme converges to the viscosity solution. This allows one to design so-called filtered schemes, which blend together high-order nonmonotone schemes with monotone first-order schemes in such a way that the resulting filtered scheme is approximately monotone. The idea at a high level is to use the higher order scheme in regions where the solution is smooth while falling back on the monotone scheme near singularities.  High order filtered schemes have received a lot of attention recently \cite{oberman2015filtered,lions1995convergence,froese2013convergent,bokanowski2016efficient,warin2013some,sahu2017high,sahu2016high}.

In this paper, we show how to construct arbitrary order filtered upwind finite schemes for the Hamilton-Jacobi equation \eqref{P1}. The upwind direction for this Hamilton-Jacobi equation is to look backwards along the coordinate axes. Therefore, our upwind schemes all use backward difference quotients in order to follow the flow of the characteristics and select the viscosity solution. This allows the schemes to be solved in a single pass yielding fast (linear complexity) algorithms. We prove that our filtered schemes are stable and convergent for any order, and we present numerical simulations investigating rates of convergence.

Our numerical simulations lead to two surprising conclusions that merit future work. First, we observe that backward differences (without filtering) for the two dimensional version of \eqref{P1} appear to be stable for order $k\leq 2$, borderline stable for $k=3$ (numerical solutions remain bounded but do not converge), and highly unstable for $k\geq 4$. We recall that it is a classical fact that backward differences for the one-dimensional version of \eqref{P1} (e.g., $u'(x)=f(x)$) are stable for order $k\leq 6$ and unstable for $k\geq 7$. It would be interesting to prove the order $k=2$ scheme is stable for \eqref{P1}, and examine the situation in higher dimensions. We note the filtered higher order schemes are stable for any $k$.

Second, we observe that filtering higher order schemes with first order monotone schemes is only successful at increasing accuracy for order $k=2$. For order $k\geq 3$, we find that the filtering relies too much on the first order scheme, and while the schemes are stable and convergent, the order of accuracy is closer to first order. This is true even when the solution of \eqref{P1} is smooth. As far as we are aware, this observation has not appeared in the literature on filtered schemes; the existing literature~\cite{oberman2015filtered,lions1995convergence,froese2013convergent,bokanowski2016efficient,warin2013some,sahu2017high,sahu2016high} has only considered numerical experiments with second order schemes. This observation refutes the conventional wisdom that one can filter any higher order scheme---the choice of the higher order scheme may be crucial, and would be an interesting problem to pursue in future work.

We should note that while some filtered schemes show higher order convergence rates in some test cases, there are no proofs that any schemes have convergence rates better than first order in general. This is a limitation in the viscosity solution theory; in fact, since solutions are not classical, the best provable rate in general is $O(\sqrt{h})$, with a one-sided $O(h)$ rate when the solution is semi-concave~\cite{caldernotes}. Even in the special case where the solution of the Hamilton-Jacobi equation is smooth, it is generally difficult to prove a higher order convergence rate, since it requires a strong stability result for the higher order scheme and the maximum principle is unavailable.

The coefficients of the difference quotients in our schemes can be obtained by solving a small linear system involving a Vandermonde matrix, as is common in the literature. As an interesting addition to the paper, we give explicit formulas for the coefficients for arbitrary order backward  difference quotients, and give simple direct proofs of the formulas. We expect these coefficients have appeared explicitly before in the literature, but we include the results for completeness. Our method extends to computing the coefficients for nonsymmetric and offset centered differences, which we explore in Section \ref{sec:diff}. We do not explicitly invert the Vandermonde matrix, but this could be done via an LU factorization as was given in ~\cite{turner1966vandermonde}.

In \cite{calder2015PDE} a fast approximate nondominated sorting algorithm was developed based on estimating the distribution of the data, which is the right-hand side $f$, and solving the Hamilton-Jacobi equation \eqref{P1} numerically. The algorithm is called PDE-based ranking and was shown to be significantly faster than nondominated sorting in low dimensions. The higher order filtered schemes developed in this paper can be directly used in the PDE-based ranking algorithm to improve the accuracy with minimal additional computational cost. Thus, this work has the potential to have a broad impact in applications of nondominated sorting.

This paper is organized as follows. In Section \ref{sec:scheme} we present our filtered schemes and prove stability and convergence. In Section \ref{sec:sims} we present the results of numerical simulations, and in Section \ref{sec:diff} we give explicit formulas for the coefficients of various difference quotients.

\section{Filtered schemes}
\label{sec:scheme}

In this section we introduce our filtered schemes and prove stability and convergence.

 \subsection{ Backward differences}
 \label{sec:bfd}

 Since the upwind direction for \eqref{P1} is the negative orthant, our higher order schemes will all use backward difference quotients. The following theorem gives the exact coefficients for all backward different quotients for first derivatives. We expect this is known in the literature, but we give the proof for completeness.

\begin{theorem}\label{BackwardDiff}
 Let $k$ be a positive integer  and  $f\in C^{k +1}$. Then
$$f'(x) = \frac{1}{h}\sum_{i=1}^{k}c_i\left[f(x-ih) - f(x)\right] + O(h^k)$$
where $$c_i = \frac{(-1)^{i}}{i}\binom {k}{i} \ \ \text{for } i = 1, \dots, k.$$
\end{theorem}
 The proof relies on  an elementary lemma, which is useful to state independently.\begin{lemma}\label{BinomialPowerZero}
For any given integer $m$ and positive integer $n \geq 2,$
$$\sum_{i=0}^{n} (i+m)^k\binom {n}i (-1)^i = 0,$$
for all $k=1,2, \dots, n-1.$
\end{lemma}
\begin{proof}
Let
$$ q_0(x) = x^m(x-1)^{n} = \sum_{i=0}^{n}\binom {n}i x^{i+m}(-1)^{n-i}.$$
For $k \geqslant 1,$ define $$q_k(x) = x\frac{dq_{k-1}(x)}{dx}.$$
So, $$ q_k(x) = \sum_{i=0}^{n}(i+m)^k\binom {n}i x^{i+m}(-1)^{n-i}.$$
It's can also be observed that for all $k \leq n-1, q_k(1) = 0.$\\
Therefore, $$\sum_{i=0}^{n} (i+m)^k\binom {n}i (-1)^i = (-1)^{-n}q_k(1) = 0.$$
\end{proof}
 We now give the proof of Theorem \ref{BackwardDiff}.
\begin{proof}
By using a Taylor series expansion, we obtain
\begin{equation}\label{TaylorExpansion2}
\sum_{i=1}^{k}c_i\left[\sum_{j=1}^{k}\frac{(-ih)^j}{j!} f^{(j)}(x)\right]  + O(h^{k+1})= \sum_{i=1}^{k}c_i\left[f(x-ih) - f(x).\right]
\end{equation}
Rearranging the left hand side of equation \eqref{TaylorExpansion2}, we deduce
$$\sum_{i=1}^{k}c_{i}\left[\sum_{j=1}^{k}\frac{(-ih)^j}{j!} f^{(j)}(x)\right]  + O(h^{k+1}) = \sum_{j=1}^{k}\frac{h^j}{j!}f^{(j)}(x) \left[\sum_{i=1}^{k}c_{i}\cdot (-i)^j \right]  + O(h^{k+1}).$$
By binomial expansion, $$\sum_{i=1}^{k}c_{i}\cdot (-i)  = -\sum_{i = 1}^{k}\binom {k}{i}(-1)^{i} = \binom {k}{0} - \sum_{i = 0}^{k}\binom {k}{i}(-1)^{i}= 1.$$
By Lemma \ref{BinomialPowerZero}, for $j = 2, \dots, k,$
$$\sum_{i=1}^{k}c_{i} (-i) ^j = (-1)^{j}\sum_{i = 1}^{k} i^{j-1} \binom {k}{i}(-1)^{i} =(-1)^{j}\sum_{i = 0}^{k} i^{j-1} \binom {k}{i}(-1)^{i}= 0.$$
Thus,
$$\sum_{i=1}^{k}c_{i}\left[\sum_{j=1}^{k}\frac{(-ih)^j}{j!} f^{(j)}(x)\right] = hf'(x) + O(h^{k+1}).$$
Therefore, we yield $$f'(x) = \frac{1}{h}\sum_{i=1}^{k}c_i\left[f(x-ih) - f(x)\right] + O(h^k)$$ as desired.
\end{proof}
 We can also express the backward difference quotient in  a more usual form.
\begin{corollary}\label{ExplicitBackward}
 Let $k$ be a positive integer  and $f\in C^{k +1}$. Then
$$f'(x) = \frac{1}{h}\sum_{i=0}^{k}d_i \cdot f(x-ih) + O(h^k)$$
where 
\begin{eqnarray*}
d_0 &=& 1 + \frac{1}{2} + \cdots + \frac{1}{k}, \ \ \text{ and}\\
d_i &=& \frac{(-1)^{i}}{i}\binom {k}{i}, \ \ \text{for } i =1, \dots, k.\qedhere
\end{eqnarray*}
\end{corollary}

\begin{proof} It's easy to see that $d_i = c_i$ (defined as in Theorem \ref{BackwardDiff}) for $i \neq 0.$ In the case when $i = 0,$ we have
\begin{eqnarray*}
d_0 &=& -\sum_{i=1}^{k} d_i \\
&=&  \sum_{i=1}^{k} \frac{(-1)^{i-1}}{i}\binom {k}{i}\\
&=&  \int_0^1\sum_{i=1}^k \binom {k}{i}(-x)^{i-1} dx \\
&=& \int_0^1\frac{1-(1-x)^k}{x} dx\\
&=& \int_0^1 1 + (1-x) + \cdots + (1-x)^{k-1} dx\\
&=& \left[ x - \frac{(1-x)}{2}  - \frac{(1-x)^2}{3} - \cdots - \frac{(1-x)^k}{k} \right]_{x=0}^{x=1} \\
&=& 1 + \frac{1}{2} + \cdots + \frac{1}{k}.
\end{eqnarray*}
\qedhere
\end{proof}

 \subsection{High-order  filtered finite difference schemes}

The Hamilton-Jacobi equation \eqref{P1} does not have smooth or even Lipschitz solutions, due to a gradient singularity near the boundary where $x_i=0$. Indeed, in the special case that $f=1$ the viscosity solution is $u=n(x_1\cdots x_n)^{1/n}$. Following \cite{calder2017} we first perform a singularity factorization before solving the Hamilton-Jacobi equation. In particular, let $u$ be the viscosity solution of \eqref{P1} and define 
\begin{equation}
 w (x) =\frac{ u ( x)}{n \left( x_{1}\cdots x_{n} \right)^{1/n}}. 
\label{eq:factor}
\end{equation}
 It is possible to show \cite{calder2017} that $w$ is Lipschitz continuous and satisfies
\begin{equation}\label{eq:wlip}
[w]_{1;[0,1]^n} \leq \sqrt{n}[f^{1/n}]_{1;[0,1]^n}.
\end{equation}
 We can also show that  $w$ is the unique bounded viscosity solution of the Hamilton-Jacobi equation
\begin{equation}\label{P3}
\prod_{i=1}^n (w + nx_i w_{x_i}) = f\ \ \text{ on } (0,1]^n.
\end{equation}
  Although it appears that  \eqref{P3} does not have a boundary condition, it is in fact  encoded into the Hamilton-Jacobi equation, since $w (0) ^{n} = f (0)$.
 See \cite{calder2017} for more details on the above two statements.  The idea from \cite{calder2017}, which we borrow here, is to solve \eqref{P3} numerically and then undo the transformation  \eqref{eq:factor} to obtain a numerical approximation of  $u$.  The work  in \cite{calder2017} explored first order finite difference schemes for \eqref{P3}  and  proved optimal convergence rates of $ O (\sqrt{h})$.  We extend these ideas here to higher order filtered schemes.

 In light of Corollary \ref{ExplicitBackward} we define for  $u:\left[0,1  \right] ^{n}\lto  \R$  the $k ^{\rm th}$-order  backward difference quotient to be
\begin{equation}
\dqki u(x) =\frac{1}{h}\sum_{j=0}^{k}d_j u(x-jhe_i).
\label{eq:backwarddifference}
\end{equation}
Here, $h>0$ is the grid resolution, and as a convention we take $u (x) = 0$ whenever $x\inn  \left[ 0,1 \right]^{n}$.  The boundary value is irrelevant and does not enter into the scheme.  By  Corollary  \ref{ExplicitBackward} we have that
\[\dqki u(x) = u_{x_i}(x) + O(h^k)\]
 whenever  $u\in C^{ k +1}$.

 We define the first order upwind scheme  approximating  the left-hand side  of \eqref{P3}  to be
\begin{equation}
 F_{1} \left( x,w \right) = 
\begin{cases}
  \displaystyle\prod_{i=1}^{n}\left( w(x)+nx_i \dqi{1} w(x) \right), &  \text{if } \forall i,\,w(x)+nx_i \dqi{1} w(x)\geq 0\\
  -\infty,& \text{otherwise.}
\end{cases}
\label{eq:F1}
\end{equation}
 The first order scheme from \cite{calder2017} corresponds to solving
 \[  F_{1}( x, w)= f ( x) \ind [0,1]^{n}_{h},\]
 where $\Omega_h:=\Omega \cap h\Z^n$  for any $\Omega\subset\R^n$.  This scheme has a unique solution, and is monotone, stable, and convergent to the viscosity solution. Furthermore,  this  scheme can be solved with a (linear time) fast sweeping algorithm visiting each grid point exactly once. For more details on this scheme,  we refer the reader to \cite{calder2017}.
 
We define the $k ^{\rm th }$-order upwind  finite difference scheme  approximating the left-hand side of  \eqref{P3} by
\begin{equation}
 F_{k} \left( x,w \right) = \prod_{i=1}^{n}\left( w(x)+nx_i \dqki w(x) \right).
\label{eq:Fk}
\end{equation}
 The $k ^{\rm th }$-order upwind  finite difference scheme  is then given by
\begin{equation}
  F_{k}( x, w)= f ( x) \ind [0,1]^{n}_{h}.
\label{eq:Fk2}
\end{equation}
 We note that this scheme is not monotone and may not have a unique solution. We can construct a solution with a fast sweeping method, as  we did with the first  order scheme, but there is no guarantee that the scheme will be stable or convergent to the viscosity solution.  In Section \ref{sec:sims} we present results of simulations suggesting that the second order scheme is stable and convergent, but higher order schemes are unstable. 

The $k ^{\rm th }$-order \emph{filtered} upwind  finite difference  approximation of the left-hand side of \eqref{P3} is given by
\begin{equation}
 G_k \left( x,w \right) = 
\begin{cases}
  F_k(x,w), &  \text{if } |F_k(x,w)-F_1(x,w)|\le \sqrt{h} \text{ and } x\in [kh,1]^n ,\\
  F_1(x,w),& \text{otherwise.}
\end{cases}
\label{eq:Gk}
\end{equation}
 The $k ^{\rm th }$-order \emph{filtered} upwind  finite difference scheme  is then given by
\begin{equation}
  G_{k}( x, w)= f ( x) \ind [0,1]^{n}_{h}.
\label{eq:Gk2}
\end{equation}
  The idea behind the filtering is that when the solution is smooth we should have  $|F_k-F_1|\leq Ch$ and so the higher order scheme is selected. In regions where the solution is not smooth, the filtered scheme falls back on the monotone and stable first order scheme.  The key property of filtered schemes is that any solution  $w_h$ of \eqref{eq:Gk2} also satisfies
\begin{equation}
 f(x) - \sqrt{h} \leq F_{1}( x, w_h) \leq f ( x)+ \sqrt{h} \ind [0,1]^{n}_{h}.
\label{eq:filtProperty}
\end{equation}
  By approximately solving the first order  scheme, the solutions of the filtered schemes inherit the stability and convergence properties of the first-order scheme.  We note that solutions of \eqref{eq:Gk2} may not be unique, since the scheme is not monotone. We compute the solution as follows: At each  grid point we solve both of the schemes $F_1=f$ and $F_k=f$ for $w_h(x)$, taking the largest root when there are multiple solutions.  We then take the solution of $F_k=f$ and check if the first property in \eqref{eq:Gk} is satisfied.  If so,  we choose this solution for  $w_h(x)$,  otherwise we select the solution of $F_1=f$.  We analyze stability and convergence in the next section.

  \subsection{Stability and convergence}

 We prove here stability and convergence of the filtered schemes given in \eqref{eq:Gk2}.
\begin{theorem}
 Assume that $f\ge 0$ and let  $w_{h}$  be a solution of \eqref{eq:Gk2}.  Then we have
\begin{equation}
0\le w_h \le  \left( \max_{[0,1]^n}f + \sqrt{h}\right) ^{1/n}  \ind [0,1]^n_h.
\label{eq:stability}
\end{equation}
\label{thm:stability}
\end{theorem}
\begin{proof}
At the minimum of $w_h$, we have $x_i \dqi{1} w_h (x) \leq 0$. Since $w_h+ n x_i\dqi{1} w_h (x) \geq 0$ for all $i = 1, 2, \dots, n,$ it follows that 
\[\min{w_h} \geq - n x_i\dqi{1} w_h (x)\geq 0.\]
This establishes a lower bound of $w_h.$

To determine an upper bound of $w_h$, we note that the filtered scheme satisfies
\[0 \leq F_1(x,w_h)  \leq f(x) + \sqrt{h}.\]
At the maximum of $w_h$, we have $x_i\dqi{1} w_h(x) \geq 0.$ Since $0 \leq w_h(x),$ it follows that
\[\left(\max{w_h}\right)^n \leq \prod_{i=1}^{n}\left( w(x)+nx_i \dqki w(x) \right) \leq f(x) +\sqrt{h} \leq \max{f} +\sqrt{h}.\]
Thus $\max{w_h} \leq (\max{f} +\sqrt{h})^{1/n}$.
This establishes an upper boundary of $w_h.$
Therefore, 
\[0\le w_h \le  \left( \max_{[0,1]^n}f + \sqrt{h} \right) ^{1/n}  \ind [0,1]^n_h,\]
and the stability follows.
\end{proof}

 Given the stability result in Theorem \ref{thm:stability},  it is now a standard application of the Barles-Souganidis  framework~\cite{barles1991}  to prove convergence of the filtered scheme \eqref{eq:Gk2} to the viscosity solution of \eqref{P3}.  In fact,  the proof of convergence  of the filtered schemes is very similar to the results in \cite{calder2017}.  Let us mention, however, that these convergence results do not establish any convergence rate. The best provable convergence rate for the filtered scheme is still $O (\sqrt{ h})$.  In practice, we do often see far better convergence rates for the filtered schemes, and we present simulations investigating convergence rates in Section \ref{sec:sims}.

\section{Numerical simulations}
\label{sec:sims}

We run simulations on both backward difference schemes and filtered schemes of orders 1, 2, 3, 5, 8, and 13 with two probability density functions $f_1$ and $f_2$ that were introduced before in \cite{calder2017}. The function $f_1$ is defined as follows:
\[f_1(x) = \frac{1}{4(k+1)^2}\prod_{i=1}^2 \left(\sum_{j=1}^2 \sin(kx_j)^2 + 2k + 2kx_i\sin(2kx_i)\right),\]
where $k>0$. In the simulations, we set $k=20$. The solution of \eqref{P1} in this case is known to be
\[u_1(x) = \frac{1}{k+1}\sqrt{x_1x_2} \left(\sin(kx_1)^2 + \sin(kx_2)^2+ 2k\right).\]
We note that the solution $u_1$ is smooth on $(0,1]^2$.

The function $f_2$ is defined as follows:
\[f_2(x) = \frac{1}{(C+2)^2}\left(w_2(x)  + 2(1+C)x(2)\right)\left( w_2(x) + 2x(1)\right).\]
where $x(i) = x_{\pi_x(i)}$ for a permutation $\pi_x$ such that $x(1)\leq x(2)$, and
\[w_2(x) =C\max\{x_1, x_2\} +  x_1 + x_2.\]
We set $C=10$ in the simulations. The solution in this case is known to be
\[u_2(x) = 2\sqrt{x_1x_2} w_2(x).\]

The function $f_2$ is Lipschitz and the solution $u_2$ has a gradient discontinuity. This is common in Hamilton-Jacobi equations due to crossing characteristics.

Given these $f_1$ and $f_2$, we gather the errors from their numerical solutions compared to the known solutions for each order of each scheme in different mesh sizes $h$. These errors are measured in both the $L^1$ norm and the $L^\infty$ norm as numerical evidence of the rate of convergence. The results from the simulations are shown in Figure \ref{fig:F2_d1} - Figure \ref{fig:F3_d4}. Note that here we use ``S'' to denote ``backward difference scheme'' and ``FS'' to denote ``filtered scheme''. 

\begin{figure}
\centering
\includegraphics[scale=0.65]{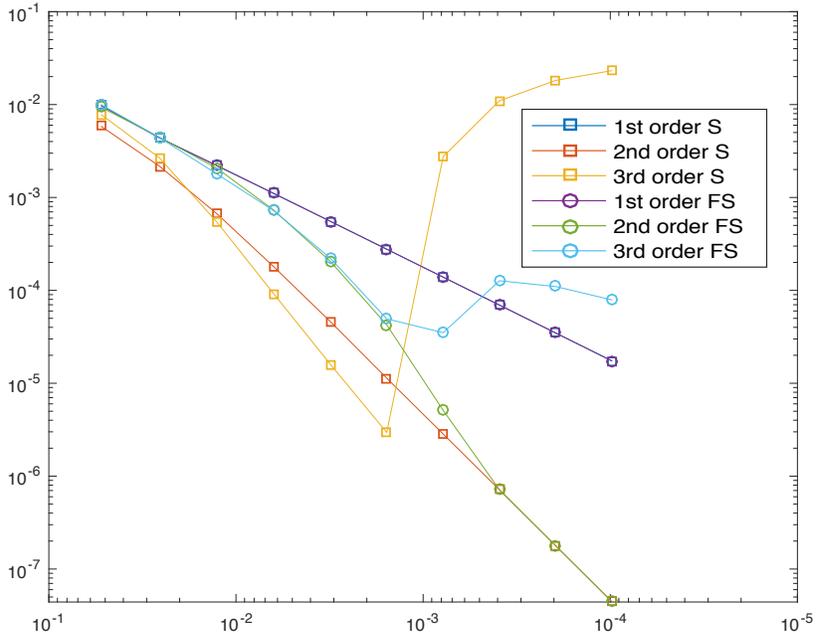}
\caption{Errors in the $L^1$ norm from order 1, 2, 3 schemes when $f = f_1$}\label{fig:F2_d1}
\end{figure}

\begin{figure}
\centering
\includegraphics[scale=0.65]{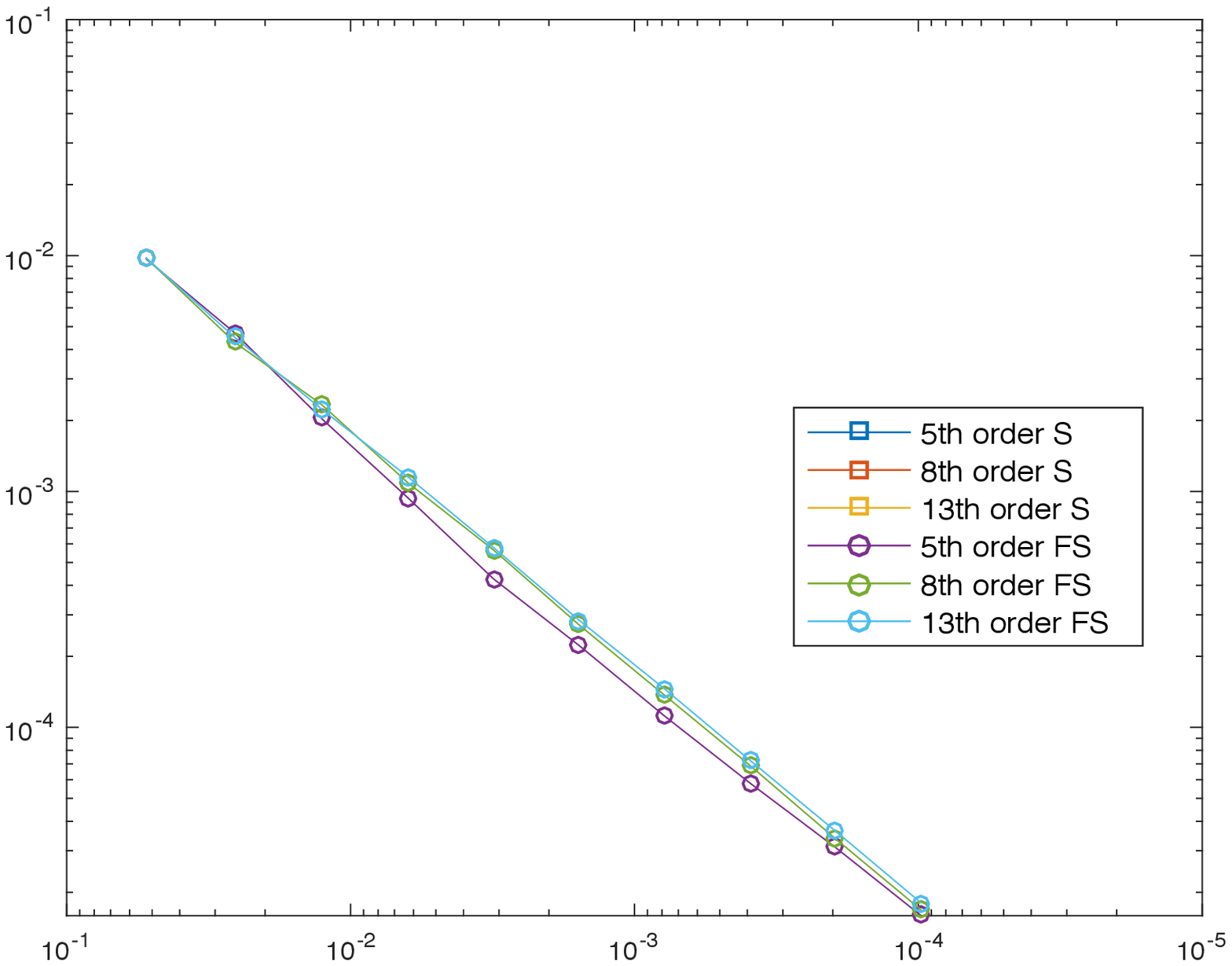}
\caption{Errors in the $L^1$ norm from order 5, 8, 13 schemes when $f = f_1$}\label{fig:F2_d2}
\end{figure}

\begin{figure}
\centering
\includegraphics[scale=0.65]{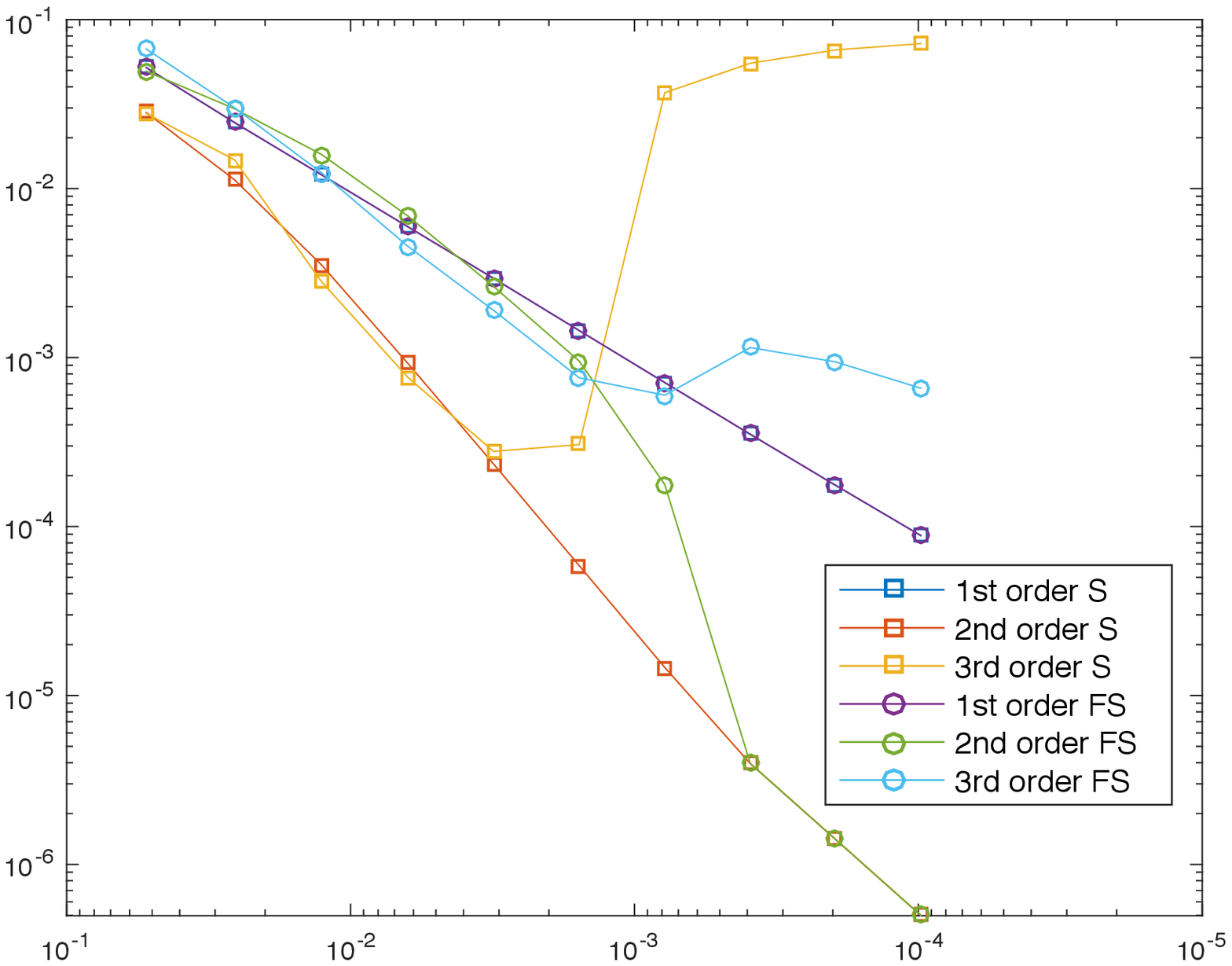}
\caption{Errors in the $L^\infty$ norm from order 1, 2, 3 schemes when $f = f_1$}\label{fig:F2_d3}
\end{figure}

\begin{figure}
\centering
\includegraphics[scale=0.65]{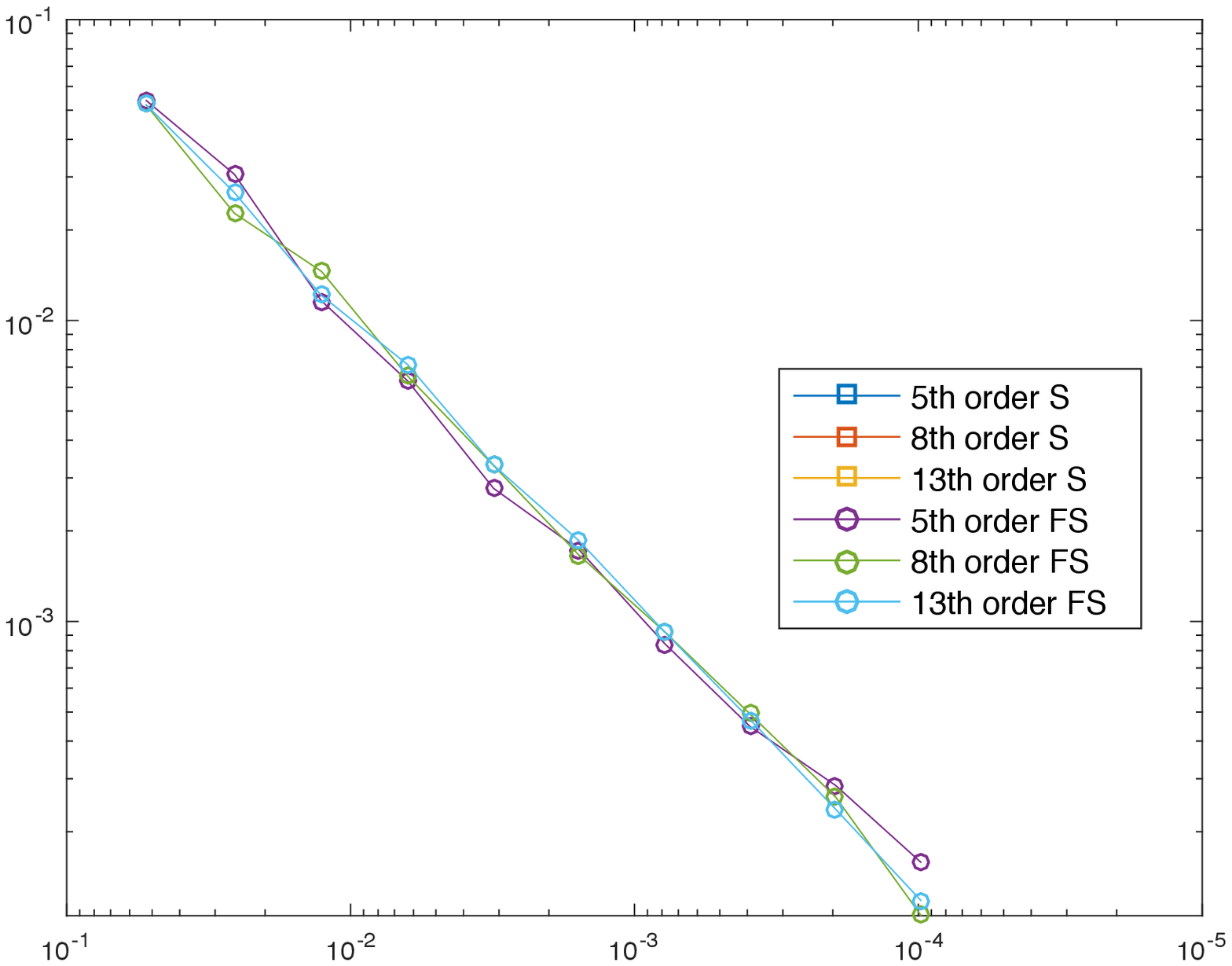}
\caption{Errors in the $L^\infty$ norm from order 5, 8, 13 schemes when $f = f_1$}\label{fig:F2_d4}
\end{figure}

\begin{figure}
\centering
\includegraphics[scale=0.65]{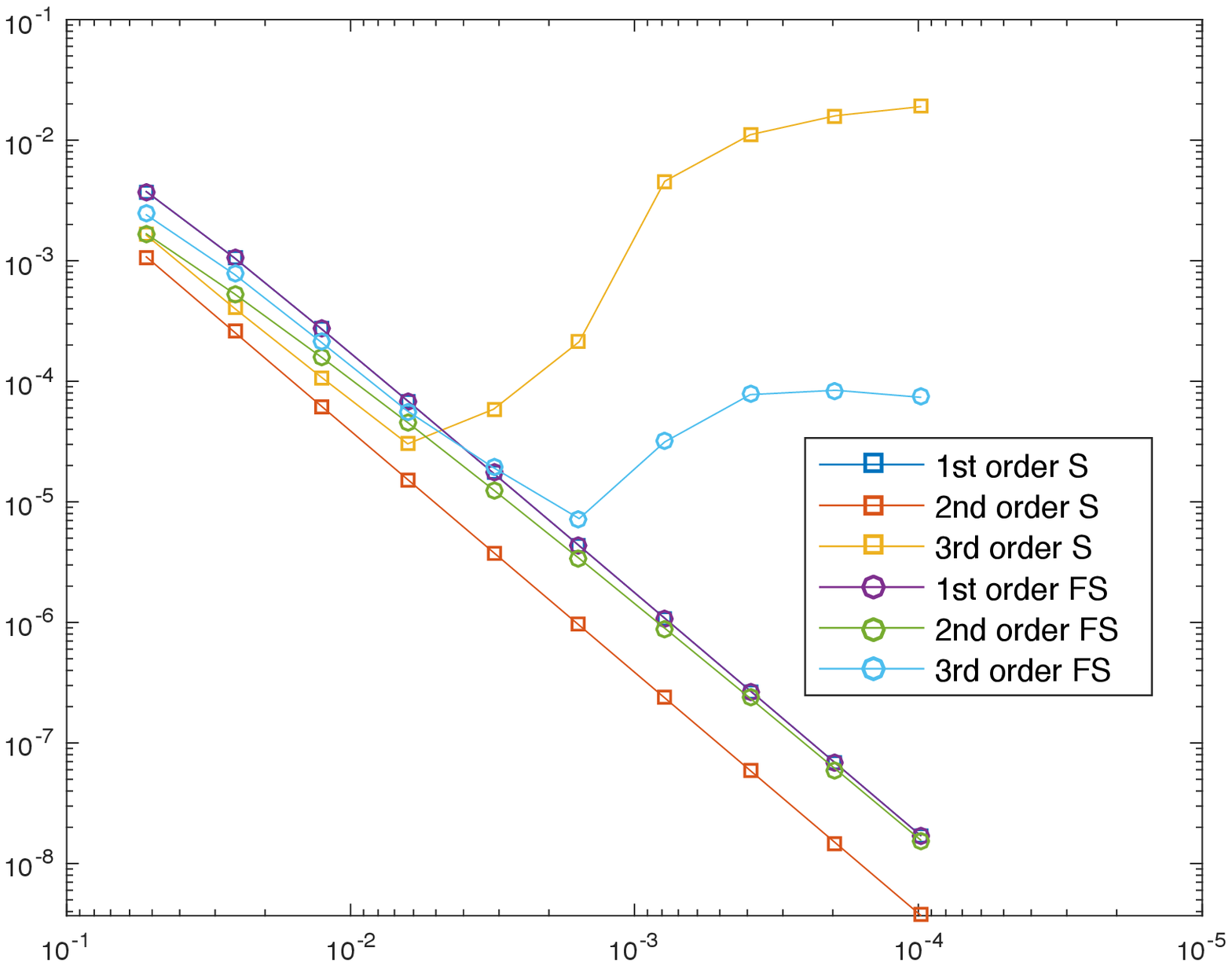}
\caption{Errors in the $L^1$ norm from order 1, 2, 3 schemes when $f = f_2$}\label{fig:F3_d1}
\end{figure}

\begin{figure}
\centering
\includegraphics[scale=0.65]{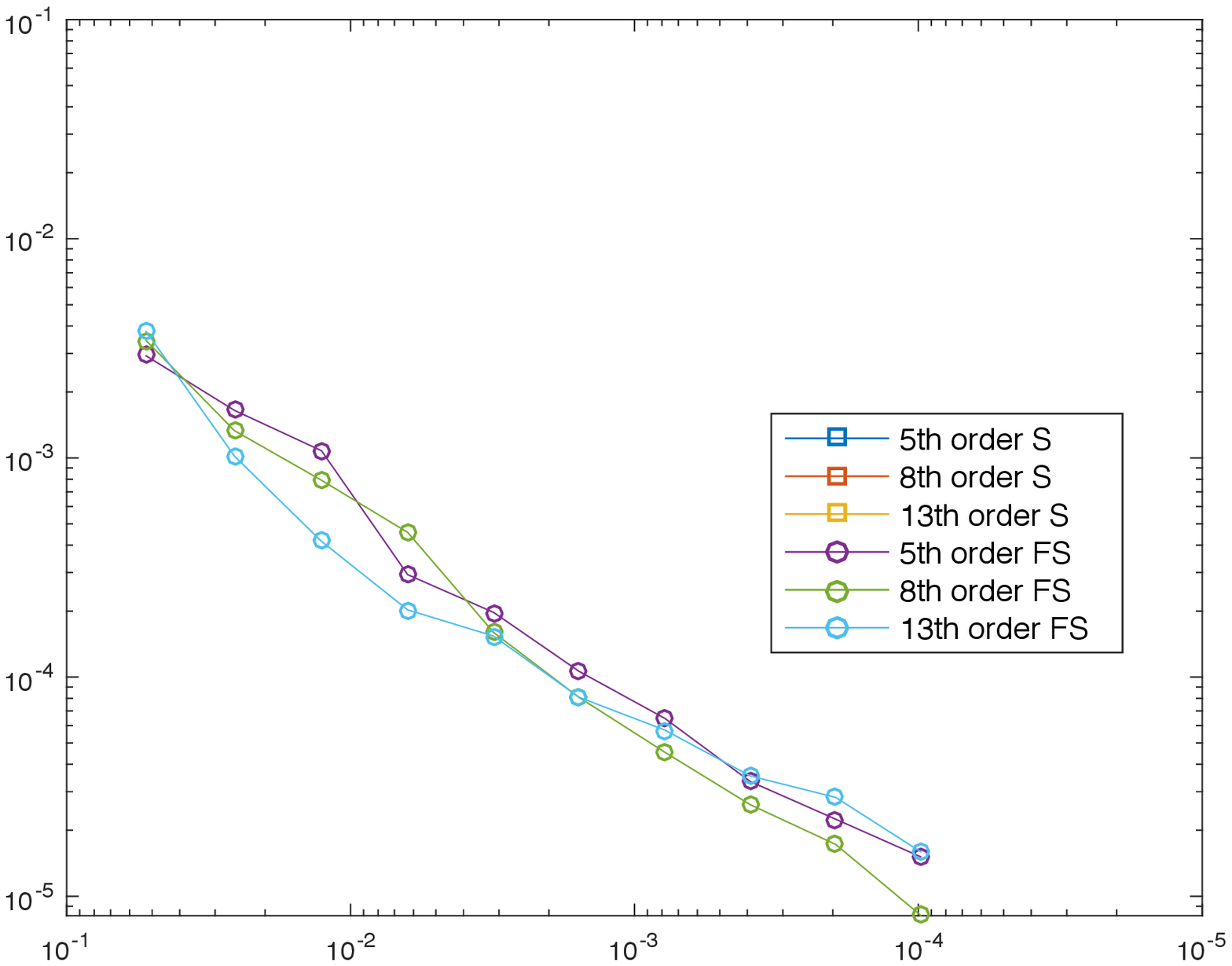}
\caption{Errors in the $L^1$ norm from order 5, 8, 13 schemes when $f = f_2$}\label{fig:F3_d2}
\end{figure}

\begin{figure}
\centering
\includegraphics[scale=0.65]{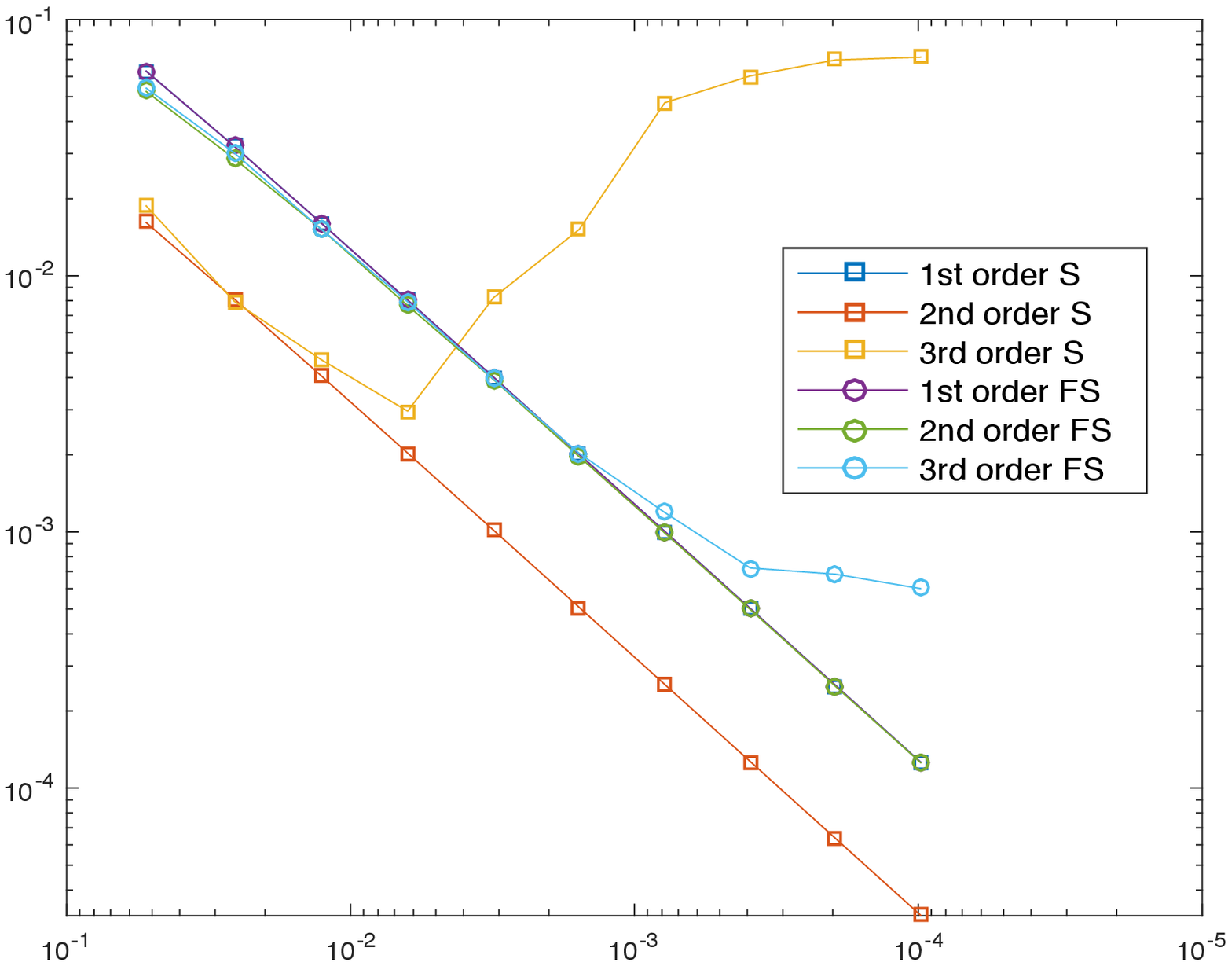}
\caption{Errors in the $L^\infty$ norm from order 1, 2, 3 schemes when $f = f_2$}\label{fig:F3_d3}
\end{figure}

\begin{figure}
\centering
\includegraphics[scale=0.65]{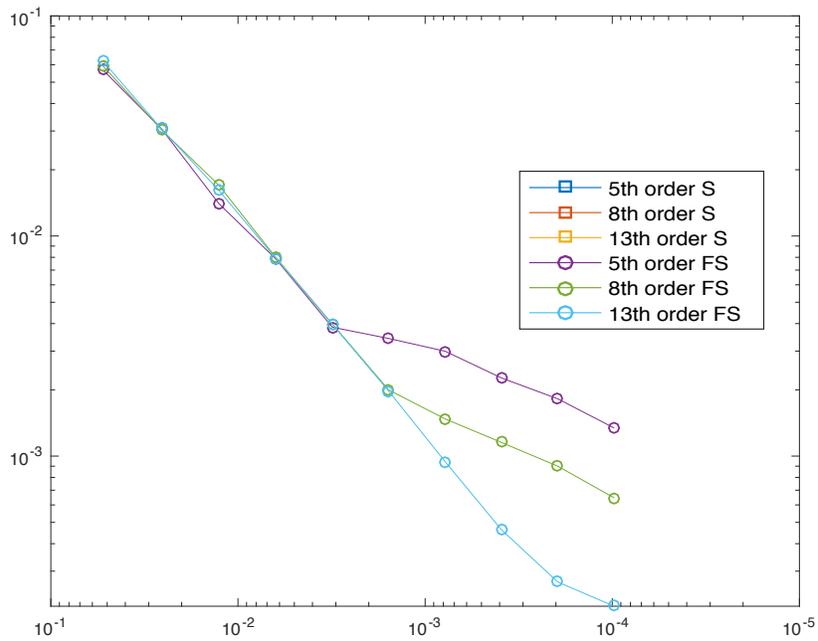}
\caption{Errors in the $L^\infty$ norm from order 5, 8, 13 schemes when $f = f_2$}\label{fig:F3_d4}
\end{figure}

The results from the simulations suggest that the unfiltered $2^{\rm nd}$ order backward difference scheme is convergent with a second order rate.  On the other hand, the backward difference schemes of order higher than two appear to be unstable. In fact, the errors for the unfiltered schemes for order $k=5,8,13$ are so large they are not shown in the figures.
 
Observing the errors from filtered schemes in both the $L^1$ norm and the $L^\infty$ norm, we see that the $2^{\rm nd}$ order filtered scheme also tends to give better accuracy than the $1^{\rm st}$ order one, but other higher order filtered schemes only give comparable accuracies to the $1^{\rm st}$ order scheme. A further investigation shows that high order filtered schemes rely most of the time on the first order scheme in solving \eqref{P1}. This explains why higher order filtered schemes do not produce better accuracy than lower order ones. To give a better idea, we show the fraction of grid points for which the $k^{\rm th}$ order scheme is being used for various mesh sizes and orders when setting $f=f_1$ in Table \ref{f2L1}. Since setting $f = f_2$ gives a similar result, we omit showing the same table of data in this case. This explains why filtering is not successful for higher order schemes. It would be interesting to determine why this is happening and whether it can be improved by using different schemes or a different type of filtering.

\begin{table}[!t]
\centering
\begin{tabular}{|c|c|c|c|c|c|c|}
\hline
Mesh size $h$& $1^{\rm st}$ order& $2^{\rm nd}$ order& $3^{\rm rd}$ order& $5^{\rm th}$ order& $8^{\rm th}$ order & $13^{\rm th}$ order\\
\hline
$3.33\times10^{-2}$&$95.06\%$ &$54.31\%$ &$45.00\%$ &$44.81\%$&$22.69\%$ &$4.88\%$ \\
$6.67\times10^{-3}$&$98.75\%$ &$81.09 \%$   &$81.61\%$ &$75.41\%$&$33.95\%$ &$7.59\%$\\
$1.59\times10^{-3}$&$99.69\%$ &$97.69\%$   &$97.42\%$ &$85.61\%$  &$37.92\%$ &$8.78\%$\\
$3.92\times10^{-5}$&$99.92\%$ &$99.84\%$   &$98.57\%$ &$87.29\%$  &$39.15\%$ &$8.97\%$\\
$9.78\times10^{-6}$&$99.98\%$ &$99.96\%$   &$97.82\%$ &$87.03\%$  &$39.43\%$ &$9.00\%$\\
\hline
\end{tabular}
\caption{Fraction of grid points for which the $k^{\rm th}$ order scheme being used in the filtered schemes.}
\label{f2L1}
\end{table}


\section{Explicit formulas for difference quotients}
\label{sec:diff}

Our proof of the backward difference formulas in Section \ref{sec:bfd} can be extended to more general difference quotients. We present these results in this section.  The first part will present results regarding the coefficients of finite difference quotients for $f'$. The second part will be devoted for generalizing the coefficients of finite difference quotients for $f^{(n)}$ where $n$ is any positive integer. Although we expect some of these coefficients to have appeared in the literature before, we include this section for completeness of the paper. Additionally, we hope to provide different approaches for proving the formulas that may be simpler and more direct.

\subsection{Other finite differences for $f'(x)$}
We begin by introducing the notation for an extension of binomial coefficient.
\begin{definition} \label{nchooseK}
Let $r$ be a real number and $k$ be a positive integer. Define 
$$\binom {r}{0} = 1 \text{ and } \binom {r}{k} = \frac{r\cdot (r-1) \cdots (r - (k-1))}{k!}.$$
\end{definition}

We found that Lemma \ref{BinomialPowerZero} can be used for computing the $k-$order forward difference quotients as well. Compared to backward differences, these quotients simply have opposite signs as stated in Theorem \ref{ForwardDiff} and Corollary \ref{ExplicitForward}.

\begin{theorem}[Forward Difference]
 Let $k$ be a positive integer. Then
$$f'(x) = \frac{1}{h}\sum_{i=1}^{k}c_i\left[f(x+ih) - f(x)\right] + O(h^k),$$
where $$c_i = \frac{(-1)^{i-1}}{i}\binom {k}{i},$$
for $i = 1,2, \dots, n.$
\end{theorem}\label{ForwardDiff}

\begin{corollary}\label{ExplicitForward}
Let $k$ be a positive integer. Then
$$f'(x) = \frac{1}{h}\sum_{i=0}^{k}d_i\cdot f(x+ih) + O(h^k),$$
where $$d_0 = -1 - \frac{1}{2} - \cdots - \frac{1}{k}, \text{ and}$$
$$d_i = \frac{(-1)^{i-1}}{i}\binom {k}{i},$$
for $i = 1,2, \dots, n.$
\end{corollary}

The proofs for both of the above theorem and corollary will be omitted as they can be proceeded in the same way as the proofs for Theorem \ref{BackwardDiff} and Corollary \ref{ExplicitBackward}, respectively.\\

Another finite difference quotient that is often implemented is centered difference quotient. With Lemma \ref{BinomialPowerZero}, we can also generalize centered difference quotients as shown in Theorem \ref{CenteredDiff}.
	
\begin{theorem}[Centered Difference]
Let $m$ and $n$ be a positive integers. Then
$$f'(x) =\frac{1}{h}\sum_{i=-m}^{n}c_i\left[f(x+ih) - f(x)\right] + O(h^{n+m}),$$
where $c_0 = 0,$ and
$$c_i = (-1)^{i-1}\frac{\binom {n+m}{i+m}}{i\binom {n+m}{m}},$$ 
for $i = -m, -m + 1, \dots, -1, 1, \dots, n.$
\end{theorem}\label{CenteredDiff}

\begin{proof}
By using a Taylor series expansion, we obtain
\begin{equation}\label{TalorExpansion3}
\sum_{i=-m}^{n}c_{i}\left[\sum_{j=1}^{n+m}\frac{(ih)^j}{j!} f^{(j)}(x)\right]  + O(h^{n+m+1})= \sum_{i=-m}^{n}c_{i}\left[f(x+ih) - f(x)\right].
\end{equation}
We claim that $c_0 = 0$ and
$$c_i = (-1)^{i-1}\frac{\binom {n+m}{i+m}}{i\binom {n+m}{m}},$$ for $i \neq 0$, reduce the left hand side of equation \eqref{TalorExpansion3} to $hf'(x) + O(h^{k+1}).$\\
Rearranging the left hand side of equation \eqref{TalorExpansion3}, we deduce
$$\sum_{i=-m}^{n}c_{i}\left[\sum_{j=1}^{n+m}\frac{(ih)^j}{j!} f^{(j)}(x)\right]  + O(h^{n+m+1}) = \sum_{j=1}^{k}\frac{h^j}{j!}f^{(j)}(x) \left[\sum_{i=-m}^{n}c_{i}\cdot i^j \right]  + O(h^{n+m+1}).$$
By binomial expansion, 
\begin{eqnarray*}
\sum_{i=-m}^{n}c_{i}\cdot i  &=& \frac{1}{\binom {n+m}{m}} \left[ \sum_{i = -m}^{n}\binom {n+m}{i+m}(-1)^{i-1} + \binom{n+m}{m} \right]\\
&=& \frac{1}{\binom {n+m}{m}} \left[ \sum_{i = 0}^{n+m}\binom {n+m}{i}(-1)^{i} + \binom{n+m}{m} \right] \ \ \ \ = 1.
\end{eqnarray*}
By Lemma \ref{BinomialPowerZero}, for $j = 2, \dots, n+m,$
\begin{eqnarray*}
\sum_{i=-m}^{n}c_{i}\cdot i ^j &=& \frac{1}{\binom {n+m}{m}} \left[ \sum_{i = -m}^{n}i^{j-1}\binom {n+m}{i+m}(-1)^{i-1} \right] \\
&=& \frac{1}{\binom {n+m}{m}} \left[ \sum_{i = 0}^{n+m}(i-m)^{j-1}\binom {n+m}{i}(-1)^{i} \right]\ \ \ \ \ = 0.
\end{eqnarray*}
Thus,
$$\sum_{i=-m}^{n}c_{i}\left[\sum_{j=1}^{n+m}\frac{(ih)^j}{j!} f^{(j)}(x)\right]  + O(h^{n+m+1}) = hf'(x) + O(h^{n+m+1}),$$ and the desired result follows.
\end{proof}

The next finite difference quotient that will be introduced in Theorem \ref{ArithmeticType1} uses arithmetic progression differences in computing $f'(x).$ Essential lemmas for proving its coefficients are given below.

\begin{lemma}\label{FractionalSum1}
For any given real number $r$ and positive integer $p$,
$$\sum_{i = 0}^{p}\binom{-r}{i} \binom{r + p}{p-i} = 1.$$
\end{lemma}

\begin{proof}
Writing
$$(x+1)^p = (x+1)^{-r} (x+1)^{(r + p) } = \left[ \sum_{i = 0}^{\infty} \binom{-r}{i} x^i\right] \left[ \sum_{j = 0}^{\infty} \binom{r+p}{j} x^j\right],$$
one can compare the coefficient of $x^p$ on both sides of the equation and yield
$$1 = \sum_{i = 0}^{p}\binom{-r}{i} \binom{r+ p}{p-i}.$$
\end{proof}

\begin{corollary}\label{G0}
For any given real number $r$ and positive integer $p$, the polynomial
$$G_0(\lambda) = \sum_{i = 0}^{p}\binom{-r + \lambda}{i}\binom{r + p - \lambda}{p-i} = 1.$$
\end{corollary}

\begin{lemma}\label{zeroCoefficient2}
For any given real number $r$ and positive integer $p,$\\
define $s_{r}^{[0]}(x)=x^r(x+1)^{-r}$, and for $k\geq 1,$
$$s_{r}^{[k]}(x)=x\frac{d\left[s_{r}^{[k-1]}(x) \right]}{dx}.$$
Then for $k = 1, 2, \dots, p$,
$$Q_{r,p}^{[k]}(x)=x^{-r} \cdot (x+1)^{r +p} \cdot  s_{r}^{[k]}(x)$$
is a polynomial of degree less than $p.$ 
\end{lemma}

\begin{proof}
We will show, by mathematical induction, that for all $k\geq 1,$ $$s_{r}^{[k]}(x)=g(x)\cdot x^r \cdot (x+1)^{-(r+k)},$$ where $g(x)$ is a polynomial of degree less than $k$.\\
It is true for $k =1$ because $$s_{r}^{[1]}(x)=r\cdot x^r\cdot(x+1)^{-(r+1)}.$$
Let's assume that this is true for some positive integer $l.$ Then we will have that
\begin{eqnarray*}
s_{r}^{[l+1]}(x)&=&x\frac{d\left[s_{r}^{[l]}(x) \right]}{dx}\\
&=& x\left[ \left[ g'(x)\cdot x^r + rx^{r-1}\cdot g(x)\right] \cdot (x+1)^{-(r+l)} \right]\\
&\ \ \ & -x\left[ (r+l)(x+1)^{-(r+l+1)}\cdot g(x) \cdot x^r\right]\\
&=& \left[\left[ x g'(x) +rg(x)\right] \cdot (x+1) - (r+l)x g(x)\right]\cdot x^r \cdot (x+1)^{-(r+l+1)}.
\end{eqnarray*}
Since $g(x)$ has degree less than $l$, it is not difficult to see that $$\left[ x g'(x) +rg(x)\right] \cdot (x+1) - (r+l)x g(x)$$ has degree less than $l+1.$ Hence, $g(x)$ has degree less than $k$ for all $k\geq 1.$\\
Therefore, for $k = 1, 2, \dots, p,$ $$Q_{r,p}^{[k]}(x)=x^{-r}(x +1)^{r + p} \cdot s_{r}^{[k]}(x) = (x+1)^{p-k} g(x)$$
is a polynomial of degree less than $p.$
\end{proof}

\begin{lemma}\label{FractionalZeroSum}
For any given real number $r$ and positive integer $p$,
$$\sum_{i = 0}^{p} (r+i)^k \binom{-r}{i} \binom{r + p}{p-i} = 0,$$
for $k = 1, 2, \dots, p.$
\end{lemma}

\begin{proof}
Let $s_{r}^{[0]}(x)=x^r(x+1)^{-r}$, and for $k\geq 1,$
$$s_{r}^{[k]}(x)=x\frac{d\left[s_{r}^{[k-1]}(x) \right]}{dx}.$$
By letting,
$$Q_{r,p}^{[k]}(x)=x^{-r} \cdot (x+1)^{r+p} \cdot  s_{r}^{[k]}(x),$$
one will have that the coefficient of $x^{p}$ is
$$\sum_{i = 0}^{p} (r+i)^k \binom{-r}{i} \binom{r + p}{p-i} .$$
By Lemma \ref{zeroCoefficient2}, for $k = 1, 2, , \dots, p$ , the polynomial $Q_{r,p}^{[k]}(x)$ is of degree less than $p.$
Therefore, 
$$\sum_{i = 0}^{p} (r+i)^k \binom{-r}{i}  \binom{r + p}{p-i} =0.$$
\end{proof}

\begin{theorem}[Arithmetic Progression Difference]
 Let $a$ be a real number, $d$ be a nonzero real numbers, and $k$ be a positive integer such that $0 \notin \lbrace a_1,a_2, \dots, a_k \rbrace,$ where $a_i = a + d(i-1).$ Then 
$$f'(x) = \frac{1}{h}\sum_{i=1}^{k}c_{a_i}\left[f(x+a_i h) - f(x)\right] + O(h^k),$$
where $$c_{a_i} = \frac{1}{a_i}\binom {-a_1/d}{i-1} \binom {a_k/d}{k-i},$$
for $i = 1,2, \dots, k.$
\end{theorem}\label{ArithmeticType1}

\begin{proof}
By using a Taylor series expansion, we obtain
\begin{equation}\label{TalorExpansion}
\sum_{i=1}^{k}c_{a_i}\left[\sum_{j=1}^{k}\frac{(a_i h)^j}{j!} f^{(j)}(x)\right]  + O(h^{k+1})= \sum_{i=1}^{k}c_{a_i}\left[f(x+a_i h) - f(x)\right].
\end{equation}
We claim that $$c_{a_i} = \frac{1}{a_i}\binom {-a_1/d}{i-1}  \binom {a_k/d}{k-i} $$ reduces the left hand side of equation \eqref{TalorExpansion} to $hf'(x) + O(h^{k+1}).$\\
Rearranging the left hand side of equation \eqref{TalorExpansion}, we deduce
$$\sum_{i=1}^{k}c_{a_i}\left[\sum_{j=1}^{k}\frac{(a_i h)^j}{j!} f^{(j)}(x)\right]  + O(h^{k+1}) = \sum_{j=1}^{k}\frac{h^j}{j!}f^{(j)}(x) \left[\sum_{i=1}^{k}c_{a_i}\cdot a_i^j \right]  + O(h^{k+1}).$$
By Lemma \ref{FractionalSum1}, $$\sum_{i=1}^{k}c_{a_i}\cdot a_i  = \sum_{i = 1}^{k}\binom{-a_1/d}{i-1}  \binom{a_k/d}{k-i}  = \sum_{i = 0}^{k-1}\binom{-a/d}{i} \binom{a/d + (k-1)}{k-i-1} = 1.$$
By Lemma \ref{FractionalZeroSum}, for $j = 2, \dots, k,$
\begin{eqnarray*}
\sum_{i=1}^{k} c_{a_i} \cdot a_i ^j &=& \sum_{i = 1}^{k} a_i^{j-1} \binom{- a_1 /d}{i-1} \binom{a_k /d}{k-i}  \\
&=& d^{j-1}\cdot \sum_{i = 0}^{k-1} {(a/d+i)}^{j-1} \binom{-a/d}{i} \binom{a/d + (k-1)}{k-i-1} = 0.
\end{eqnarray*}
Thus,
$$\sum_{i=1}^{k}c_{a_i}\left[\sum_{j=1}^{k}\frac{(ih)^j}{j!} f^{(j)}(x)\right]  + O(h^{k+1}) = hf'(x) + O(h^{k+1}),$$
and the desired result follows.
\end{proof}

\subsection{General forms of finite differences for $f^{(n)}$}

We extend the idea from previous results to generalize the coefficients of finite difference quotients for $f^{(n)}$.  In computing them, we need to solve for the inverses of the matrices defined as in Theorem \ref{matrix1} and \ref{matrix2}. Because of its generality, the results presented in this section will also hold for all the previous results. Indeed, the first columns of these inverses represent the coefficients of a finite difference quotient for $f'$, the second columns represent the coefficients of a finite difference quotient for $f''$, and so on. Some notations and lemmas that are essential for proving them are given below.

\begin{lemma}\label{PowerLeadingCoef}
For any given real numbers $r, \lambda,$ and positive integer $p,$\\
define $t_{\lambda, r}^{[0]}(x)=x^r(x+1)^{-(r-\lambda)}$, and for $k\geq 1,$
$$t_{\lambda, r}^{[k]}(x)=x\frac{d\left[t_{\lambda, r}^{[k-1]}(x) \right]}{dx}.$$
Then for $k = 1, 2, \dots, p$, the coefficient of the term $x^{p}$ in the polynomial
$$R_{\lambda, r,p}^{[k]}(x)=x^{-r} \cdot (x+1)^{(r-\lambda) +p} \cdot  t_{\lambda, r}^{[k]}(x)$$
is $\lambda^k.$ 
\end{lemma}

\begin{proof}
We will show by mathematical induction that for all $k\geq 1,$ $$t_{\lambda, r}^{[k]}(x)=g(x)\cdot x^r \cdot (x+1)^{-r+\lambda+k},$$ where $g(x)$ is a polynomial of degree $k$ with the leading coefficient of $\lambda^k$.\\
It is true when $k =1$ because $$t_{\lambda, r}^{[1]}(x)=(\lambda x + r)\cdot x^r \cdot(x+1)^{-r+\lambda+1}.$$
Let's assume that this is true for some positive integer $l.$ Then we will have that
\begin{eqnarray*}
t_{\lambda, r}^{[l+1]}(x)&=&x\frac{d\left[t_{\lambda, r}^{[l]}(x) \right]}{dx}\\
&=& x\left[ \left[ g'(x)\cdot x^r + rx^{r-1}\cdot g(x)\right] \cdot (x+1)^{-r+\lambda+l} \right]\\
&\ \ & -x\left[ \left[ r-\lambda+l\right] (x+1)^{-r+\lambda+l+1}\cdot g(x) \cdot x^r\right]\\
&=& \left[\left[ x g'(x) +rg(x)\right] \cdot (x+1) - (r - \lambda +l)x\cdot g(x)\right]\cdot x^r \cdot (x+1)^{-r+\lambda+l+1}.
\end{eqnarray*}
Since the leading term of $g(x)$ is $\lambda^l x^{l}$, it follows that the leading coefficient of \\
$\left[ x g'(x) +rg(x) \right] \cdot (x+1) - (r - \lambda +l)x\cdot g(x)$ is 
$$\left[ (l +r)\lambda^l - (r-\lambda + l)\lambda^l \right] x^{l+1} = \lambda^{l+1} x^{l+1}.$$ 
Hence, $g(x)$ is a polynomial of degree $k$ with the leading coefficient of $\lambda^k$ for all $k\geq 1.$\\
Therefore, for $k = 1, 2, \dots, p,$ the coefficient of the term $x^{p}$ in the polynomial $$R_{\lambda, r,p}^{[k]}(x)=x^{-r}(x+1)^{-r+\lambda + p} \cdot t_{\lambda, r}^{[k]}(x)$$
is $\lambda^k.$
\end{proof}

\begin{lemma}\label{FractionalPolynomial}
Let $r$ be a real number and $p$ be a positive integer. Define the polynomial
$$H_i(\lambda) = \binom{-r + \lambda}{i}  \binom{r + p - \lambda}{p-i}.$$
Then, for $i = 1, \dots, p,$
\[
H_i(\lambda ^*) = 
\begin{cases}
1 \text{ if } \lambda  ^*= r + i,\\
0 \text{ if } \lambda ^* = r + j, \text{ where } 0 \leq j \leq p \text{ and } j \neq i.
\end{cases}
\]
\end{lemma}

\begin{proof} In the case when $\lambda = r + i,$ we directly compute
$$H_i(r+i) = \binom{-r + (r+i)}{i} \binom{r + p - (r+i)}{p-i} = \binom{i}{i} \binom{p - i}{p-i} = 1.$$
In other cases, we express $H_i$ as a product of linear functions of $\lambda$ and yield
\begin{eqnarray*}
\binom{-r + \lambda}{i} \binom{r + p - \lambda}{p-i} &=& H_i(\lambda)\\
&=& (-1)^{p-i}\frac{(\lambda-r)\cdot (\lambda-(r+1))\cdots (\lambda-(r+p))}{(\lambda-(r+i))\cdot i!\cdot (p-i)!}.
\end{eqnarray*}
From here, it is difficult to see that for $i = 1, \dots, p,$
$$H_i(\lambda ^*) = 0 \text{ if } \lambda ^* = r + j, \text{ where } 0 \leq j \leq p \text{ and } j \neq i.$$
\end{proof}

\begin{lemma}\label{GK}
For any given real number $r$ and positive integer $p$, the polynomial
$$G_k(\lambda) = \sum_{i = 0}^{p} (r+i)^k \binom{-r + \lambda}{i} \binom{r + p - \lambda}{p-i}= \lambda^k,$$
for $k = 1, 2, \dots, p.$
\end{lemma}

\begin{proof} In the similar fashion to the proof for Lemma \ref{FractionalZeroSum}, we have that, for $k = 1, 2, \dots, p,$ the expression
$$\sum_{i = 0}^{p} (r+i)^k \binom{-r + \lambda}{i} \binom{r + p - \lambda}{p-i}$$ is the coefficient of $x^{p}$ in the polynomial $R_{\lambda, r, p}^{[k]}(x)$ from the previous theorem. Hence, $$\sum_{i = 0}^{p} (r+i)^k \binom{-r + \lambda}{i} \binom{r + p - \lambda}{p-i} = \lambda^k.$$ Thus, it automatically holds when we define the polynomial
$$G_k(\lambda) = \sum_{i = 0}^{p} (r+i)^k \binom{-r + \lambda}{i} \binom{r + p - \lambda}{p-i}.$$
\end{proof}

\begin{definition}\label{Sigma}
Let $S = \lbrace s_1,s_2, \dots, s_n \rbrace$ be a set containing real numbers. Define $\sigma_0(S) = 1$ and $$\sigma_k(S) = \sum_{a_1<a_2<\cdots<a_k \in S} a_1\cdot a_2 \cdot \cdots a_k,$$
for $i = 1, 2, \dots, n.$
\end{definition}

\begin{lemma} \label{matrix1}
For any real number $a,$ nonzero real number $d,$ and positive integer $n$ such that $0 \notin T = \lbrace a, a+d, \dots, a+(n-1)d \rbrace$, define $A$ to be an $n\times n-$matrix with
$$A_{ij} = (a+(j-1)d)^i.$$ Then $A^{-1} = B,$ where $$B_{ij} = \frac{(-1)^{i+j}}{a+(i-1)d} \cdot \frac{\sigma_{n-j}(T - \lbrace a+(i-1)d\rbrace)}{d^{n-1}(i-1)!(n-i)!}.$$
\end{lemma}

\begin{proof}
We will first show that $AB = I_n.$ By directly computing the product of $A$ and $B$, we deduce
\begin{eqnarray*}
(AB)_{ij} &=& \sum_{k=1}^n A_{ik} \cdot B_{kj}\\
&=& \sum_{k=1}^n (a+(k-1)d)^i\cdot \left[ \frac{(-1)^{k+j}}{a+(k-1)d} \cdot \frac{\sigma_{n-j}(T - \lbrace a+(k-1)d\rbrace)}{d^{n-1}(k-1)!(n-k)!}\right]\\
&=& \sum_{k=1}^n (-1)^{k+j}(a+(k-1)d)^{i-1} \cdot \frac{\sigma_{n-j}(T - \lbrace a+(k-1)d\rbrace)}{d^{n-1}(k-1)!(n-k)!}\\
&=& d^{i-j}\cdot \sum_{k=0}^{n-1} (-1)^{k+j+1}(a/d+k)^{i-1} \cdot \frac{\sigma_{n-j}(T - \lbrace a+kd\rbrace)}{d^{n-j}(k)!(n-k-1)!}.
\end{eqnarray*}
We note that $(AB)_{ij}$ equals the coefficient of $x^{j-1}$ in the polynomial $d^{i-j}\cdot  G_{i-1}(\lambda),$ where $G(\lambda)$ is defined as in Corollary \ref{G0} and Lemma \ref{GK} (by setting $r = a/d$) . Thus, by Corollary \ref{G0} and Lemma \ref{GK},
\[
(AB)_{ij}=\begin{cases}
1 & \text{if } i = j,\\
0  & \text{if } i \neq j.
\end{cases}
\]
This establishes $AB = I_n$ as desired. \\
Now we will show that $BA = I_n.$ By directly computing the product of $B$ and $A$, we deduce
\begin{eqnarray*}
(BA)_{ij}&= &\sum_{k=1}^n B_{ik} \cdot A_{kj}\\
&=& \sum_{k=1}^n \left[ \frac{(-1)^{i+k}}{a+(i-1)d} \cdot \frac{\sigma_{n-k}(T - \lbrace a+(i-1)d\rbrace)}{d^{n-1}(i-1)!(n-i)!}\right](a+(j-1)d)^k\\
&=& \frac{a+(j-1)d}{a+(i-1)d}\left[\sum_{k=1}^n (-1)^{i+k}(a+(j-1)d)^{k-1} \cdot \frac{\sigma_{n-k}(T - \lbrace a+(i-1)d\rbrace)}{d^{n-1}(i-1)!(n-i)!}\right]\\
&=& \frac{a+(j-1)d}{a+(i-1)d} \left[\sum_{k=1}^n (-1)^{i+k}(a/d+j-1)^{k-1} \cdot \frac{\sigma_{n-k}(T - \lbrace a+(i-1)d\rbrace)}{d^{n-k}(i-1)!(n-i)!}\right]\\
&=&\frac{a+(j-1)d}{a+(i-1)d}  \cdot H_i(a/d + j-1),
\end{eqnarray*}
where $H_i$ is defined as in Lemma \ref{FractionalPolynomial}.
Thus, by Lemma \ref{FractionalPolynomial},
\[
(BA)_{ij}=\begin{cases}
1 & \text{if } i = j,\\
0  & \text{if } i \neq j.
\end{cases}
\]
This establishes $BA = I_n.$ Therefore, $A^{-1} = B.$
\end{proof}

\begin{lemma} \label{matrix2}
For any positive integers $m, n,$ and nonzero real number $d$, let the set $T = \lbrace -md, -(m-1)d, \dots, -d, d, \dots, nd\rbrace$ and define the $(m+n)\times (m+n)-$matrix
$$A = 
\begin{bmatrix}
C \ D
\end{bmatrix}
$$
where $C$ is an $(m+n) \times m-$matrix with $C_{ij} = \left[(j-m-1)d \right]^i$ and $D$ is an $(m+n)\times n-$matrix with $D_{ij} = (jd)^i.$ Then 
$$ A^{-1} = B = 
\begin{bmatrix}
C^*  \\ D^*
\end{bmatrix}
$$
where $C^*$ is an $m\times (n+m)-$matrix with 
$$C^*_{ij} =(-1)^{i+j+1}\cdot \frac{\sigma_{n+m-j}(T - \lbrace(i-m-1)d\rbrace)}{d^j \cdot (i-1)!(n+m-i+1)!},$$ and
$D^*$ is an $n\times (n+m)-$matrix with 
$$D^*_{ij} =(-1)^{m+i+j}\cdot \frac{\sigma_{n+m-j}(T - \lbrace id\rbrace)}{d^j \cdot (m+i)!(n-i)!}.$$
\end{lemma}

The proof for Lemma \ref{matrix2} is omitted since it can be proceeded in the similar way as the proof for Lemma \ref{matrix1}. The reader who is interested in proving it may set $d =1$ for simplicity. Once the case $d =1$ is proven, it is not difficult to see that the result also holds in the case $d \neq 1$. 

It is not practical to explicitly express matrix $B$. However, it is not too difficult to express the first, the last, and the second last columns of $B$. We already show how to compute the first column in the previous two sections (aka Theorem \ref{BackwardDiff} - \ref{ArithmeticType1}). In this section, we will only show how to compute the last and the second last columns. Their proofs will be omitted due to the fact that they are derived directly from Lemma \ref{matrix1} and Lemma \ref{matrix2}.\\

\textbf{Computing the last column of $B$}

\begin{theorem}\label{lastCol1}
 Let $a$ be a real number, $d$ be a nonzero real number, and $n$ be a positive integer such that $0 \notin \lbrace a_1,a_2, \dots, a_n \rbrace,$ where $a_i = a + d(i-1).$ Then 
$$f^{(n)}(x) = \frac{1}{h^n}\sum_{i=1}^{n}c_{a_i}\left[f(x+a_i h) - f(x)\right] + O(h),$$
where $$c_{a_i} = \frac{(-1)^{n+i}}{a_i} \cdot \frac{1}{d^{n-1}(i-1)!(n-i)!},$$
for $i = 1,2, \dots, n.$
\end{theorem}

\begin{theorem}\label{lastCol2}
For any positive integers $m, n,$ and nonzero real number $d$, let $a_i = i\cdot d.$
Then 
$$f^{(m+n)}(x) = \frac{1}{h^{m+n}}\sum_{i=-m}^{n}c_{a_i}\left[f(x+a_i h) - f(x)\right] + O(h),$$
where $c_{a_0} = 0,$ and
$$c_{a_i} = (-1)^{n+i}\cdot \frac{1}{d^{n+m} \cdot (m+i)!(n-i)!},$$
for $i = -m, \dots, -1, 1, \dots, n.$\\
\end{theorem}

\textbf{Computing the second last column of $B$}

\begin{theorem}\label{secondLastCol1}
 Let $a$ be a real number, $d$ be a nonzero real number, and $n$ be a positive integer such that $0 \notin \lbrace a_1,a_2, \dots, a_n \rbrace,$ where $a_i = a + d(i-1).$ Then 
$$f^{(n-1)}(x) = \frac{1}{h^{n-1}}\sum_{i=1}^{n}c_{a_i}\left[f(x+a_i h) - f(x)\right] + O(h^2),$$
where $$c_{a_i} = \frac{(-1)^{n+i-1}}{a_i} \cdot \frac{\frac{n}{2}(a_1 + a_n) - a_i}{d^{n-1}(i-1)!(n-i)!},$$
for $i = 1,2, \dots, n.$
\end{theorem}

\begin{theorem}\label{secondLastCol2}
For any positive integers $m, n,$ and nonzero real number $d$, let $a_i = i\cdot d.$
Then 
$$f^{(m+n-1)}(x) = \frac{1}{h^{m+n-1}}\sum_{i=-m}^{n}c_{a_i}\left[f(x+a_i h) - f(x)\right] + O(h^2),$$
where $c_{a_0} = 0,$ and
$$c_{a_i} = (-1)^{n+i-1}\cdot \frac{ \frac{n^2-m^2}{2} - i }{d^{n+m-2} \cdot (m+i)!(n-i)!},$$
for $i = -m, \dots, -1, 1, \dots, n.$
\end{theorem}

\section{Conclusions}
\label{sec:conc}

In this paper, we showed how to construct filtered schemes for the Hamilton-Jacobi equation continuum limit of nondominated sorting by combining high order possibly unstable schemes with first order monotone and stable schemes. We proved that the filtered schemes are stable and convergent for all orders. We then investigated both high-order unfiltered and filtered schemes for the Hamilton-Jacobi equation by implementing both schemes for order $k=1, 2, 3, 5, 8,$ and $13$ numerically solving the equations in various mesh sizes. The errors from their numerical solutions compared to the known solutions were measured in the $L^1$ norm and the $L^\infty$ norm. Our  results suggest that the unfiltered schemes of order higher than $2$ are unstable while the $1^{\rm st}$ order and $2^{\rm nd}$ order unfiltered schemes remain stable. Moreover, we see that the $2^{\rm nd}$ order unfiltered scheme shows $2^{\rm nd}$ order accuracy.  Similarly to the unfiltered schemes, we see that the $2^{\rm nd}$ order filtered scheme seems to show $2^{\rm nd}$ order accuracy. However, it turns out that the filtered schemes of order higher than $2$ only exhibit a $1^{\rm st}$ order convergence rate. Upon further investigation, this appears to be due to fact that the filtering relies too often on the $1^{\rm st}$ order scheme. Future work would include proving stability of the second order unfiltered scheme, and investigating techniques to improve the accuracy of the higher order filtered schemes.

\end{document}